\documentclass[11pt,a4paper]{amsart}
\usepackage{graphicx,amscd}
\theoremstyle{plain}
\newtheorem{theorem}{Theorem}[section]

\newtheorem{definition}[theorem]{Definition}

\theoremstyle{remark}
\newtheorem{remark}[theorem]{Remark}
\newtheorem*{claim}{Claim}

\begin{document}
\title
{Critical Heegaard surfaces obtained by amalgamation}
\author[J. H. Lee]{Jung Hoon Lee}
\address{Department of Mathematics and Institute of Pure and Applied Mathematics,
Chonbuk National University, Jeonju 561-756, Korea}
\email{junghoon@jbnu.ac.kr}

\maketitle
\begin{abstract}
Critical surfaces are defined by Bachman as topological index $2$ surfaces,
generalizing incompressible surfaces and strongly irreducible surfaces.
In this paper we give a condition to obtain critical Heegaard surfaces by amalgamation.
As a special case, we obtain critical Heegaard surfaces by boundary stabilization.
It gives critical Heegaard surfaces of non-minimal genus,
for $3$-manifolds which do not admit distinct Heegaard splittings (up to isotopy).
\end{abstract}

\section{Introduction}

Let $S$ be a closed orientable separating surface in an irreducible $3$-manifold $M$,
dividing $M$ into two submanifolds $V$ and $W$.
Define the {\em disk complex} $\mathcal{D}_S$ as follows.
\begin{itemize}
\item Vertices of $\mathcal{D}_S$ are isotopy classes of compressing disks for $S$.
\item A collection of $k+1$ distinct vertices constitute a $k$-cell if
      there are mutually disjoint representatives.
\end{itemize}

By an abuse of terminology, we sometimes identify a vertex
with some representative compressing disk of the vertex.
Let $\mathcal{D}_S(V)$ and $\mathcal{D}_S(W)$ be the subcomplexes of $\mathcal{D}_S$
spanned by compressing disks in $V$ and $W$ respectively.
In this paper, we focus only on vertices and edges of
$\mathcal{D}_S(V)$, $\mathcal{D}_S(W)$, and $\mathcal{D}_S$.

If $S$ is an incompressible surface, then $\mathcal{D}_S$ is empty, and vice versa.
We say that $S$ is {\em strongly irreducible} if both $\mathcal{D}_S(V)$ and $\mathcal{D}_S(W)$ are non-empty,
and $\mathcal{D}_S(V)$ is not connected to $\mathcal{D}_S(W)$ in $\mathcal{D}_S$.
Strongly irreducible surfaces are proved to be useful
to analyze the Heegaard structure and topology of $3$-manifolds.
For example, if a minimal genus Heegaard surface is not strongly irreducible,
then the manifold contains an incompressible surface \cite{Casson-Gordon}.

An incompressible surface and a strongly irreducible surface can be regarded as
topological analogues of an index $0$ minimal surface and an index $1$ minimal surface respectively.
As a generalization of this idea, Bachman has defined a notion of {\em critical surface}
\cite{Bachman1}, \cite{Bachman2}, \cite{Bachman3},
which can be regarded as a topological index $2$ minimal surface.
It is equivalent to that $\pi_1(\mathcal{D}_S)$ is non-trivial.

\begin{definition}\label{def1}
A surface $S$ is critical if vertices of $\mathcal{D}_S$ can be partitioned
into two non-empty sets $C_0$ and $C_1$.

\begin{enumerate}
\item For each $i=0,1$, there is at least one pair of compressing disks
    $D_i\in \mathcal{D}_S(V)\cap C_i$ and $E_i\in \mathcal{D}_S(W)\cap C_i$
    such that $D_i\cap E_i=\emptyset$.
\item If $D\in \mathcal{D}_S(V)\cap C_i$ and $E\in \mathcal{D}_S(W)\cap C_{1-i}$,
    then $D\cap E\ne \emptyset$ for any representative disks.
    In other words, $D$ and $E$ are not joined by an edge.
\end{enumerate}
\end{definition}

Critical surfaces behave in some way similarly as
incompressible surfaces and strongly irreducible surfaces do.
For example, if an irreducible manifold contains an incompressible surface and a critical surface,
then the two surfaces can be isotoped so that any intersection loop is essential on both surfaces.
In fact, this is true for all topologically minimal surfaces \cite{Bachman3}.

Topological index theory, including critical surfaces, is a relatively new theory pioneered by Bachman.
One can ask how large the class of critical surfaces is.
In \cite{Bachman1}, it was shown that if a manifold which does not contain incompressible surfaces
has two distinct strongly irreducible Heegaard splittings,
then the minimal genus common stabilization of the two splittings is critical.
(In fact the original definition of criticality in \cite{Bachman1} is
slightly stronger than Definition \ref{def1}.
Hence the above mentioned result still holds with respect to Definition \ref{def1}.)

In this paper, we show that some critical Heegaard surfaces can be obtained by amalgamating
two strongly irreducible Heegaard splittings.

\begin{theorem}\label{thm1}
Let $X\cup_S Y$ be an amalgamation of two strongly irreducible Heegaard splittings
$V_1\cup_{S_1}W_1$ and $V_2\cup_{S_2}W_2$ along homeomorphic boundary components of
$\partial_-V_1$ and $\partial_-V_2$. See Figure \ref{fig2}.
Assume that $V_2$ is constructed from $\partial_-V_2\times I$ by attaching only one $1$-handle.
If there exist essential disks $D_1\subset W_1$ and $D_2\subset W_2$
which persist into disjoint essential disks in $Y$ and $X$ respectively, then $S$ is critical.
\end{theorem}

As a special case of Theorem \ref{thm1}, we obtain critical Heegaard surfaces by
boundary stabilization---an amalgamation with a standard type $2$ Heegaard splitting of
$({\rm surface}) \times I$ \cite{Scharlemann-Thompson}.
This gives plenty of new critical Heegaard surfaces of non-minimal genus,
also for $3$-manifolds which do not admit distinct Heegaard splittings.
In other words, it says the existence of a critical Heegaard surface of non-minimal genus
for a $3$-manifold admitting a unique minimal genus Heegaard splitting.

\begin{theorem}\label{thm2}
If a strongly irreducible Heegaard surface of a $3$-manifold with boundary
admits a disjoint pair of a vertical annulus and an essential disk on opposite sides,
then the Heegaard splitting obtained from it by a boundary stabilization is critical.
\end{theorem}

\begin{remark}\label{rmk1}
Other examples of critical Heegaard surfaces, including high index topologically minimal surfaces,
are constructed in \cite{Bachman-Johnson}.
On the other hand, unstabilized examples of critical Heegaard surfaces are shown in \cite{Lee}.
\end{remark}

Our main theorem is Theorem \ref{thm1}, but we consider Theorem \ref{thm2} first
because it is easier.
In Section $2$, we define boundary stabilization of a Heegaard splitting.
In Section $3$, we give a proof of Theorem \ref{thm2}.
In Section $4$, we define amalgamation of two Heegaard slittings.
In Section $5$, we give a proof of Theorem \ref{thm1}.

\section{Boundary stabilization}

A {\em compression body} $V$ is a $3$-manifold obtained from a closed surface $S$
by attaching some $2$-handles to $S\times\{0\}\subset S\times I$ and
capping off any resulting $2$-sphere boundary components with $3$-balls.
The surface $S\times\{1\}$ is denoted by $\partial_+ V$ and
$\partial V-\partial_+ V$ is denoted by $\partial_- V$.
If $\partial_-V=\emptyset$, then $V$ is called a {\em handlebody}.
There is also a dual description for compression body.
When $\partial_-V\ne\emptyset$, $V$ can be obtained from $\partial_-V\times I$
by attaching dual $1$-handles to $\partial_-V\times\{1\}$.
For a $3$-manifold $M$, a {\em Heegaard splitting} $V\cup_S W$ is a decomposition of $M$
into two compression bodies $V$ and $W$, where
$\partial_+V=\partial_+W=S$ and $\partial M=\partial_-V \cup \partial_-W$.
It is known that every compact $3$-manifold admits Heegaard splittings.

For a Heegaard splitting $V\cup_S W$, let $\alpha$ be a properly embedded arc in $W$
which is parallel to an arc in $S$.
We add $N(\alpha)$ to $V$ and remove it from $W$.
Then we get a new Heegaard splitting $V'\cup_{S'} W'$ with the genus increased by one.
This is called a {\em stabilization} of $V\cup_S W$.
Conversely, for a stabilized Heegaard splitting $V'\cup_{S'} W'$
there exist essential disks $D\subset V'$ and $E\subset W'$ such that $|D\cap E|=1$, and
$V'\cup_{S'} W'$ can be {\em destabilized} to a lower genus Heegaard splitting.

Let $g(\cdot)$ denote the genus of a surface.
For Heegaard splittings under consideration in Section $2$ and $3$,
we assume that $\partial_-V\ne\emptyset$ and $g(\partial_-V)<g(\partial_+V)$ and $W$ is a handlebody.

A {\em complete meridian disk system} $\{D_i\}$ for $V$ is
a collection of disjoint essential disks with $\partial D_i\subset\partial_+V$ such that
cutting $V$ along $\bigcup D_i$ results in $\partial_-V\times I$.
A {\em vertical annulus} $A$ in $V$ is a properly embedded essential annulus
with one boundary component in $\partial_-V$ and the other boundary component in $\partial_+V$.
By a standard argument, there exists a complete meridian disk system for $V$ which is disjoint from $A$.
Hence there exists a dual description $V=(\partial_-V\times[0,1])$ $\cup$ $1$-handles
in which $A$ inherits the product structure.

Take a vertical arc $\gamma$ in $A$.
A tubular neighborhood $N(\gamma)$ can be regarded as a $1$-handle
connecting $W$ to $\partial_-V\times [0,\frac{1}{2}]$.
Compression bodies $V'$ and $W'$ of a {\em boundary stabilization} $V'\cup_{S'} W'$ of
$V\cup_S W$ are defined as follows.

\begin{itemize}
\item $W'=W$ $\cup$ $N(\gamma)$ $\cup$ $(\partial_-V\times[0,\frac{1}{2}])$
\item $V'=\mathrm{cl}(M-W')$
\end{itemize}

Clearly $W'$ is a compression body.
We can see that $V'$ is homeomorphic to (punctured $\partial_-V)\times [\frac{1}{2},1]$ $\cup$ $1$-handles.
So $V'\cup_{S'} W'$ becomes a Heegaard splitting and $g(S')$ is equal to $g(S)+g(\partial_-V)$.

Boundary stabilization gives a way to obtain a higher genus Heegaard splitting
from a given Heegaard splitting of a $3$-manifold with boundary.
There are many cases that the boundary stabilized surfaces can be destabilized,
although it is not clear whether it can always be destabilized.
See (\cite{Moriah-Sedgwick}, Section $5$).
For example, a boundary stabilization of a splitting admitting
a pair of a vertical annulus and an essential disk intersecting in one point,
which is called {\em $\gamma$-primitive} in \cite{Moriah-Sedgwick},
results in a stabilized Heegaard splitting.
It is easy to see that $\gamma$-primitive Heegaard splitting satisfies
the condition of Theorem \ref{thm2}.

\section{Partition of disk complex}

In this section, we give a proof of Theorem \ref{thm2}.
Let $M=V\cup_S W$ be a strongly irreducible Heegaard splitting of a $3$-manifold with boundary
admitting a vertical annulus $A\subset V$ and an essential disk $D\subset W$ with $A\cap D=\emptyset$.
Take a vertical arc $\gamma\subset A$ and boundary stabilize $V\cup_S W$
to get $V'\cup_{S'} W'$ as in Section $2$. See Figure \ref{fig1}.
We will show that $S'$ is a critical Heegaard surface of $M$.

\begin{figure}[tbh]
\includegraphics[width=10cm]{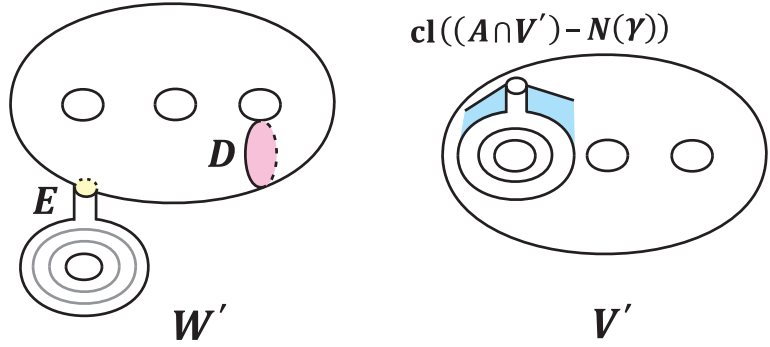}
\caption{Boundary stabilization}\label{fig1}
\end{figure}

Let $E$ be a meridian disk of $W'$ corresponding to the $1$-handle $N(\gamma)$.
Define a partition of vertices of the disk complex
$\mathcal{D}_{S'}=C_0$ $\dot{\cup}$ $C_1$ as follows.

\begin{enumerate}
\item Let $\mathcal{D}_{S'}(V')\cap C_0$ be essential disks in $V'$ that are disjoint from $E$.
\item Let $\mathcal{D}_{S'}(W')\cap C_0$ be a singleton $\{E\}$.
\item Let $\mathcal{D}_{S'}(V')\cap C_1$ be $\mathcal{D}_{S'}(V')-(\mathcal{D}_{S'}(V')\cap C_0)$.
\item Let $\mathcal{D}_{S'}(W')\cap C_1$ be $\mathcal{D}_{S'}(W')-\{E\}$.
\end{enumerate}

Since $g(\partial_-V)<g(\partial_+V)$, there are essential disks in $V$ disjoint from $\gamma$.
Hence there are essential disks in $V'$ that are disjoint from $E$.
This means that $C_0$ contains disjoint compressing disks for $S'$ on opposite sides.

Since $\gamma$ is contained in the annulus $A$, $\mathrm{cl}((A\cap V')-N(\gamma))$ is
an essential disk in $V'$ intersecting $E$ (minimally) in two points,
so it belongs to $\mathcal{D}_{S'}(V')\cap C_1$.
The essential disk $D$ persists as an essential disk in $W'$ and
belongs to $\mathcal{D}_{S'}(W')\cap C_1$.
By assumption, $\mathrm{cl}((A\cap V')-N(\gamma))$ and $D$ are disjoint.
So $C_1$ also contains disjoint compressing disks for $S'$ on opposite sides.
Hence the first condition of Definition \ref{def1} is satisfied.

By definition, any disk in $\mathcal{D}_{S'}(V')\cap C_1$ intersects $E$.
Now it remains to show that any disk in $\mathcal{D}_{S'}(V')\cap C_0$ intersects
any disk in $\mathcal{D}_{S'}(W')\cap C_1$.

Let $D_0\in\mathcal{D}_{S'}(V')\cap C_0$ be an essential disk in $V'$ that is disjoint from $E$.
Note that $V'$ is homeomorphic to (punctured $\partial_-V)\times [\frac{1}{2},1]$ $\cup$ $1$-handles,
and $\partial E$ corresponds to the puncture.
So we can see that $\partial D_0$ should be contained in $\partial_+V\cap S'$.
This means that $D_0$ can be regarded as an essential disk in $V$.
Let $E_1\in\mathcal{D}_{S'}(W')\cap C_1$ be an essential disk in $W'$ which is not isotopic to $E$.

\begin{claim}
$D_0\cap E_1\ne\emptyset$.
\end{claim}

\begin{proof}
Suppose to the contrary that $D_0\cap E_1=\emptyset$.
We assume that $E_1$ is chosen so that the number of components of intersection $|E\cap E_1|$
is minimal up to isotopy of $E_1$, satisfying $D_0\cap E_1=\emptyset$.
Since a circle component of intersection can be removed by a standard innermost disk argument,
we can assume that there is no circle component of intersection in $E\cap E_1$
Hence $E\cap E_1$ consists of arc components.
Let $\alpha$ be an outermost arc component in $E_1$ and
$\Delta$ be the corresponding outermost disk in $E_1$.
Let $\beta=cl(\partial\Delta-\alpha)$.

The disk $E$ separates $W'$ into $W$ and a manifold homeomorphic to $\partial_-V\times[0,\frac{1}{2}]$.
Suppose $\Delta$ is contained in $\partial_-V\times[0,\frac{1}{2}]$.
Surger $E$ along $\Delta$, and let $E'$ and $E''$ be the two resulting disks.
Then $E'$ is an inessential disk in $W'$ and $E''$ is isotopic to $E$, or vice versa.
In any case we can reduce $|E\cap E_1|$, a contradiction.

So $\Delta$ is contained in $W$.
Let $E'$ be one of the disk obtained by surgery of $E$ along $\Delta$.
By minimality of $|E\cap E_1|$ again, $E'$ is neither inessential in $W'$ nor isotopic to $E$.
Hence $E'$ can be regarded as an essential disk in $W$.
Since $V\cup_S W$ is a strongly irreducible Heegaard splitting, $D_0\cap E'\ne\emptyset$.
We can observe that
$$D_0\cap E'=\partial D_0\cap \partial E'=\partial D_0\cap\beta\subset \partial D_0\cap\partial E_1$$
However, $D_0\cap E_1=\emptyset$ by assumption, and this is a contradiction.
\end{proof}

Hence the partition of the disk complex $\mathcal{D}_{S'}=C_0$ $\dot{\cup}$ $C_1$ satisfies
the criticality. This completes the proof of Theorem \ref{thm2}.

\section{Amalgamation}

In this section we give a definition of amalgamation, which was first introduced in \cite{Schultens}.
Boundary stabilization discussed in Section $2$ and $3$ is an amalgamation
with a standard type $2$ Heegaard splitting of $({\rm surface})\times I$.
Let $V_1\cup_{S_1} W_1$ and $V_2\cup_{S_2} W_2$ be Heegaard splittings of 
$3$-manifolds $M_1$ and $M_2$ respectively.
Suppose $M_1$ and $M_2$ are glued together along some homeomorphic boundary components
$F_1\subset\partial_-V_1$ and $F_2\subset\partial_-V_2$.
Let $M=M_1\cup_{F} M_2$ be the resulting manifold and $F$ be the image of $F_1$ and $F_2$ in $M$.

\begin{figure}[tbh]
\includegraphics[width=12cm]{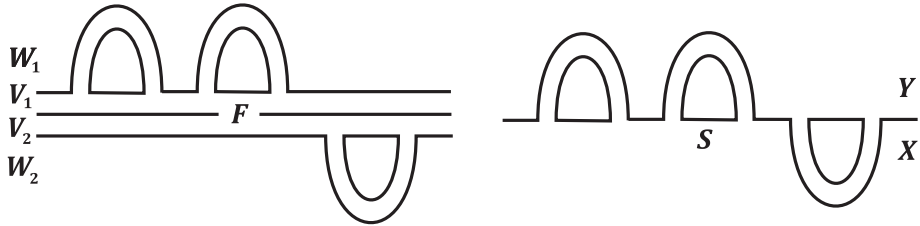}
\caption{Amalgamation}\label{fig2}
\end{figure}

By dual description of compression body, $V_1$ can be regarded as 
$(\partial_-V_1\times [0,1])$ $\cup$ $1$-handles and 
$V_2$ can be regarded as $(\partial_-V_2\times [0,1])$ $\cup$ $1$-handles.
Now collapse $(F_1\cup F_2)\times [0,1]$ to $F$ and 
regard the $1$-handles of $V_1$ and $V_2$ are attached to $F$. 
We assume that the disks where the $1$-handles are attached are mutually disjoint.
Let $X=W_2$ $\cup$ ($1$-handles in $V_1$) and $Y=W_1$ $\cup$ ($1$-handles in $V_2$). 
See Figure \ref{fig2}.
It is obvious that $X$ and $Y$ are compression bodies (or handlebodies).
The Heegaard splitting $M=X\cup_S Y$ is called an {\em amalgamation} of
$V_1\cup_{S_1} W_1$ and $V_2\cup_{S_2} W_2$.
By construction, $g(S)=g(S_1)+g(S_2)-g(F)$.
We can see that essential disks in $W_1$ and $W_2$ persist into 
essential disks in $Y$ and $X$ respectively.

\section{Proof of Theorem \ref{thm1}}
In this section, we give a proof of Theorem \ref{thm1}.
The underlying idea is same with the proof of Theorem \ref{thm2}.

Note that $V_2$ is constructed from $\partial_-V_2\times I$ by attaching only one $1$-handle.
Let $E$ be a meridian disk of $Y$ corresponding to the $1$-handle.
Define a partition of vertices of the disk complex 
$\mathcal{D}_{S}=C_0$ $\dot{\cup}$ $C_1$ as follows.

\begin{enumerate}
\item Let $\mathcal{D}_{S}(X)\cap C_0$ be essential disks in $X$ that are disjoint from $E$.
\item Let $\mathcal{D}_{S}(Y)\cap C_0$ be a singleton $\{E\}$.
\item Let $\mathcal{D}_{S}(X)\cap C_1$ be $\mathcal{D}_{S}(X)-(\mathcal{D}_{S}(X)\cap C_0)$.
\item Let $\mathcal{D}_{S}(Y)\cap C_1$ be $\mathcal{D}_{S}(Y)-\{E\}$.
\end{enumerate}

The co-core disks of $1$-handles of $V_1$ persist into essential disks in $X$ disjoint from $E$.
Hence $C_0$ contains disjoint compressing disks for $S$ on opposite sides.

The two disks $D_1$ and $D_2$ persist into disjoint essential disks in $Y$ and $X$ respectively.
Since $V_2\cup_{S_2} W_2$ is strongly irreducible, $D_2$ intersects $E$.
Hence after amalgamation, $D_1$ and $D_2$ represent elements in
$\mathcal{D}_{S}(Y)\cap C_1$ and $\mathcal{D}_{S}(X)\cap C_1$ respectively.
So $C_1$ also contains disjoint compressing disks for $S$ on opposite sides.
The first condition of Definition \ref{def1} is now satisfied.

By definition of the partition, any disk in $\mathcal{D}_{S}(X)\cap C_1$ intersects $E$.
Now it remains to show that any disk in $\mathcal{D}_{S}(X)\cap C_0$ intersects
any disk in $\mathcal{D}_{S}(Y)\cap C_1$.

\begin{figure}[tbh]
\includegraphics[width=8cm]{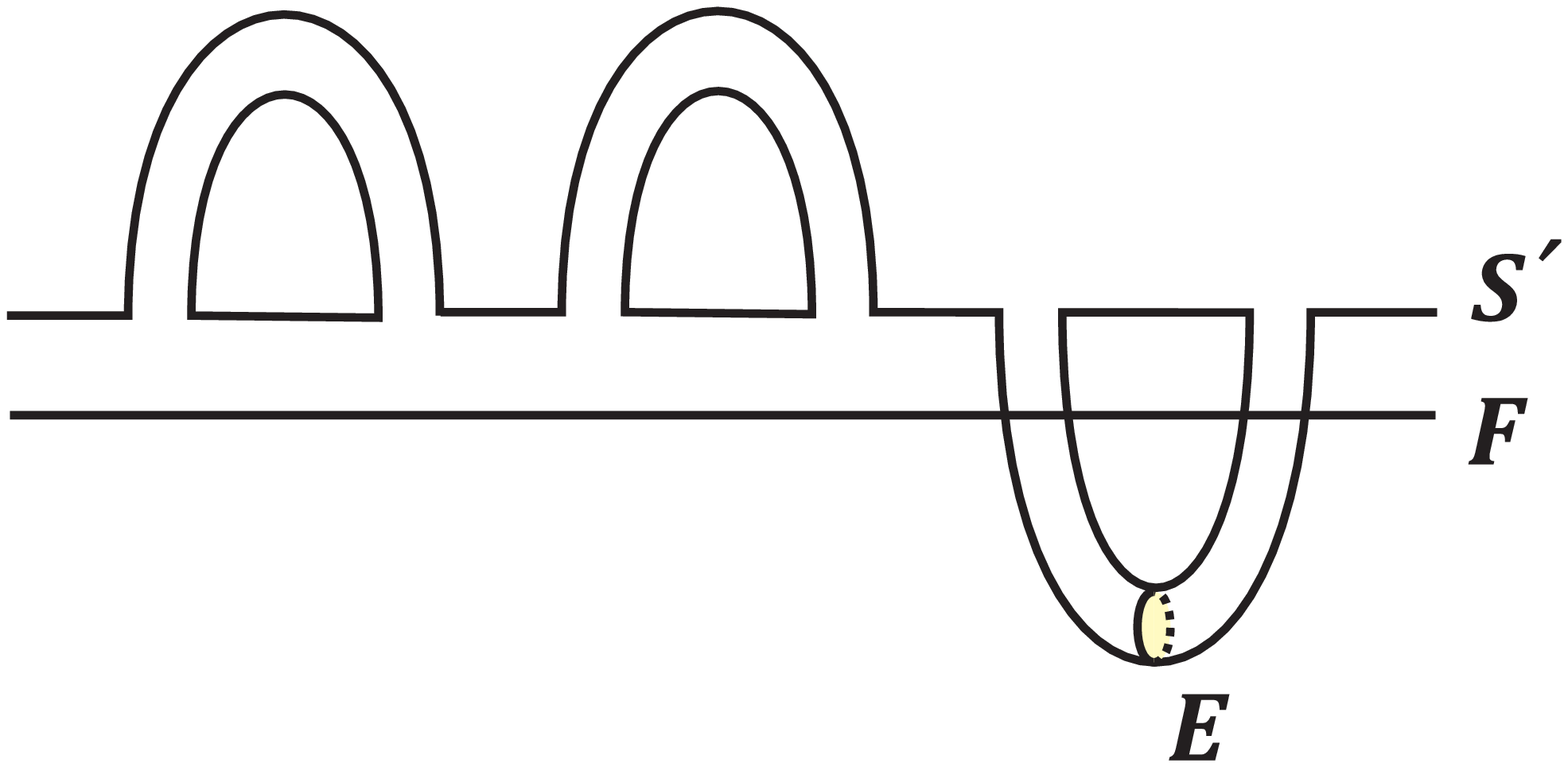}
\caption{}\label{fig3}
\end{figure}

Let $D_0\in\mathcal{D}_{S}(X)\cap C_0$ be an essential disk in $X$ that is disjoint from $E$.
The loop $\partial D_0$ is contained in a surface $S'$ homeomorphic to twice punctured $S_1$.
Take a copy of the surface $F$ so that the region between $S'$ and $F$ is
homeomorphic to $V_1$ as in the Figure \ref{fig3}.
Since both $V_1\cup_{S_1}W_1$ and $V_2\cup_{S_2} W_2$ are strongly irreducible, 
$F$ is incompressible in $M$.
Then by standard argument, we can assume that $D_0\cap F=\emptyset$.
Hence $D_0$ can be regarded as an essential disk in $V_1$.
Let $E_1\in\mathcal{D}_{S}(Y)\cap C_1$ be an essential disk in $Y$ which is not isotopic to $E$.

\begin{claim}
$D_0\cap E_1\ne\emptyset$.
\end{claim}

\begin{proof}
Suppose to the contrary that $D_0\cap E_1=\emptyset$.
We assume that $E_1$ is chosen so that the number of components of intersection $|E\cap E_1|$ is 
minimal up to isotopy of $E_1$, keeping that $D_0\cap E_1=\emptyset$.
Since a circle component of intersection can be removed by a standard innermost disk argument,
we can assume that there is no circle component of intersection in $E\cap E_1$.
Hence $E\cap E_1$ consists of arc components.
Let $\alpha$ be an outermost arc component in $E_1$ and $\Delta$ be 
the corresponding outermost disk in $E_1$. Let $\beta=cl(\partial\Delta-\alpha)$.

The disk $E$ cuts $Y$ into a manifold homeomorphic to $W_1$.
Surger $E$ along $\Delta$, and let $E'$ be one of the two resulting disks.
By minimality of $|E\cap E_1|$, $E'$ is neither inessential in $W_1$ nor isotopic to $E$.
Hence $E'$ can be regarded as an essential disk in $W_1$.
Since $V_1\cup_{S_1} W_1$ is a strongly irreducible Heegaard splitting, $D_0\cap E'\ne\emptyset$.
We can observe that
$$D_0\cap E'=\partial D_0\cap \partial E'=\partial D_0\cap\beta\subset \partial D_0\cap\partial E_1$$
However, $D_0\cap E_1=\emptyset$ by assumption, and this is a contradiction.
\end{proof}

Hence the partition of the disk complex $\mathcal{D}_{S}=C_0$ $\dot{\cup}$ $C_1$ 
satisfies the criticality. This completes the proof of Theorem \ref{thm1}.

\end{document}